\documentclass[12pt]{amsart}
\usepackage{amssymb,amsmath,amsthm,latexsym}
\usepackage{mathrsfs}
\usepackage{a4wide}
\usepackage[usenames,dvipsnames]{color}
\usepackage{euscript}
\usepackage{graphicx}
\usepackage{mdwlist}
\usepackage{mathtools,dsfont,wasysym}
\usepackage[dvipsnames]{xcolor}

\DeclareFontFamily{U}{mathx}{\hyphenchar\font45}
\DeclareFontShape{U}{mathx}{m}{n}{
      <5> <6> <7> <8> <9> <10>
      <10.95> <12> <14.4> <17.28> <20.74> <24.88>
      mathx10
      }{}

\newcommand{\nn}[1]{{\vert\kern-0.25ex\vert\kern-0.25ex\vert #1 
    \vert\kern-0.25ex\vert\kern-0.25ex\vert}}

\newtagform{scroman}[][]
 
\makeatother

\newcommand{\vertiii}[1]{{\left\vert\kern-0.25ex\left\vert\kern-0.25ex\left\vert #1 
    \right\vert\kern-0.25ex\right\vert\kern-0.25ex\right\vert}}
\renewcommand{\leq}{\leqslant}
\renewcommand{\geq}{\geqslant}
\newcounter{smallromans}

{\end{list}}

\newcounter{smallromansdash}

{\end{list}}

\newcounter{bigromans} 
  {\end{list}}

\usepackage[T1]{fontenc} 
\usepackage{url}

\usepackage{color}

\usepackage[unicode=true, pdfusetitle,
 bookmarks=true,bookmarksnumbered=true,bookmarksopen=true,bookmarksopenlevel=1,
 breaklinks=false,pdfborder={0 0 1},backref=false,colorlinks=false]{hyperref}
\hypersetup{pdfstartview=FitH}

\numberwithin{equation}{section}

\newtheorem{thm}{Theorem}[section]
\newtheorem{defi}[thm]{Definition}
\newtheorem{lemma}[thm]{Lemma}
\newtheorem{prop}[thm]{Proposition}

\newtheorem{qtn}[thm]{Question}

\theoremstyle{remark}

\theoremstyle{plain}

\newtheorem*{Mthm*}{Main Theorem}

 \DeclareMathOperator{\spn}{span}

\newcommand{\coo}{\mathrm{c}_{00}}
\newcommand{\co}{\mathrm{c}_0}

\begin{document}

\title[Monotonticity and well-separatedness]{A note on asymptotically monotone basic sequences and well-separated sets}

\author[C. S. Barroso]{C. S. Barroso$^*$}
\address{Department of Mathematics, Federal University of Cear\'a, Fortaleza 60455-360, Brazil}
\email{cleonbar@mat.ufc.br}

\thanks{$^*$ Corresponding author.}
\thanks{The research was partially supported by FUNCAP/CNPq/PRONEX Grant 00068.01.00/15.}

\date{\today}



\begin{abstract}
We remark that if $X$ is an infinite dimensional Banach space then every seminormalized weakly null sequence in $X$ has an asymptotic monotone basic subsequence. We also observe that if $X$ contains an isomorphic copy of $\ell_1$, then for every $\varepsilon>0$ there exist a $(1 +\varepsilon)$-equivalent norm $\vertiii{\cdot}$ on $X$ such that the unit sphere $(S_{(X, \vertiii{\cdot})})$ contains a normalized bimonotone basic sequences which is symmetrically $2$-separated.
\end{abstract}

\maketitle


\section{Introduction}\label{sec:intro}
It is an open problem to know whether every infinite dimensional Banach space contains a monotone basic sequence (see \cite[Problem 1.1, p. 60]{Si3}, \cite[p. 220]{FHHMZ} and \cite[p. 6]{HSVZ}). A sequence $(x_n)$ in a Banach space $X$ is called a basic sequence if it is a Schauder basis for its closed linear span, denoted by $[(x_n)]$. Recall also that the partial sums associated to a basic sequence $(x_n)$ is the sequence of projections $(P_n)$ on $[(x_n)]$ defined by $P_n\colon \sum_{i=1}^\infty a_i x_i\mapsto \sum_{i=1}^n a_i x_i$. One can prove using standard arguments in Banach space theory that $\sup_n\| P_n\|< \infty$. This number is referred in the literature as the basic constant of $(x_n)$. In \cite{Pel} Pe\l czy\'nski proved that for every $\varepsilon\in (0,1)$, every normalized weakly null sequence in $X$ has a basic subsequence with basic constant $1+\varepsilon$. This result is also known as Bessaga-Pe\l czy\'nski's selection principle, after \cite{BP} (see also, e.g., \cite[p. 14]{AK}). A basic sequence $(x_n)$ is called monotone if $\| P_n\|=1$ for all $n\in\mathbb{N}$. A bimonotone basic sequence is a monotone basic sequence so that the norm of all tail projections is exactly $1$; that is, $\| I - P_n\|=1$ for all $n\in \mathbb{N}$. Let us recall yet that $(x_n)$ is said to be asymptotically monotone if $\| P_n\|\to 1$. This slightly weaker concept of monotonicity was introduced by V. D. Milman \cite{M} (see also \cite[p. 93]{Day}). Naturally, one may also consider asymptotic bimonotonicity (i.e. $\lim_n \| I - P_n\|=\lim_n\|P_n\|=1$). 

Monotone basic sequences are of course asymptotically monotone, the former though being overall more hard to detect than the latter. In fact, a few important works have done regarding the existence of basic sequences with asymptotic monotonicity properties. For instance, it is well known that every Banach space has an asymptotically monotone basic sequence  (cf. \cite{Day}, see also \cite[Theorem 1.2, p. 49]{Si3}, \cite[Theorem 1.20]{HSVZ}). Very recently, H\'ajek, Kania and Russo \cite{HKR} obtained an useful refinement of the classical Mazur technique of constructing basic sequences, leading to a block version of this fact (see \cite[Lemma 2.4]{HKR}): every basic sequence has an asymptotically monotone block basis. This statement seems also to have been, although implicitly, witnessed by Elton and Odell in \cite[Remarks. (1)]{EO}. It is worthy to stress that asymptotic monotonicity has been proven to be very useful in many geometric and analytical problems in Banach space theory (cf. \cite{EO, HKR, MV, Si, Si2}). For instance, in \cite[Theorem 1]{EO} Elton and Odell used the fact that every Banach space contains an asymptotically monotone basic sequence in concert with Ramsey combinatorial methods to show that the unit sphere of every infinite dimensional normed space contains, for some $\varepsilon>0$, a sequence $(x_n)$ such that $\| x_n - x_m\| > 1 + \varepsilon$ for all $n\neq m$. 

With regarding asymptotic pre-monotonicity, it was open as to whether one could also obtain in every Banach space an asymptotically pre-monotone basic sequence (i.e. $\| I - P_n\|\to 1$). It turns out that this is not true in general, as pointed out by Odell and Schlumprecht in \cite[p. 1349]{OS}. Nevertheless, as a byproduct of James's non-distortion theorems it follows that every Banach space containing an isomorphic copy of $\co$, also contains an asymptotically pre-monotone basic sequence (cf. \cite[Theorem 8]{DLT}). 

A natural question to be addressed is whether a selection principle for asymptotically monotone sequences can be obtained. In this note we remark that every seminormalized weakly null sequence admits an asymptotically monotone basic subsequence. This result is probably well-known to experts but we were unable to locate a reference. It is worth noting that there is a lot of spaces containing seminormalized weakly null sequences. For example, non-Schur Banach spaces (because of Rosenthal's $\ell_1$-theorem) and infinite dimensional reflexive Banach spaces (due to the theorem of Josefson-Nissenzweig, see e.g. \cite{D}, p. 219). 

Our next goal concerns the existence of asymptotically monotone basic sequences which are symmetrically well-separated in unit spheres of Banach spaces containing isomorphic copies of $\ell_1$. This topic has been extensively dealt in several contemporary and recent works (see \cite{CP, K, HKR} and references therein). In Section \ref{sec:3}, we will prove our first result. In Section \ref{sec:4} we shall show the existence of symmetrically $2$-separated sequences of unit vectors of a Banach space with certain prescribed conditions related to monotonicity and nearly isometric renormings. In particular, we strengthen Propositions $5.1$ in \cite{HKR}. We end this note in Section \ref{sec:6} with some concluding remarks.


\section*{Acknowledgments}
A significant part of this note was done when the author was attending the AMS Joint Mathematics Meeting 2019 in Baltimore, MD. The work was completed when the author was visiting UCF -- University Central of Florida, January 15--17, 2019. He would like to thank Professors Chris Lennard and Torrey Gallagher for their invitation to give a talk at JMM-2019, Professor Eduardo Teixeira for his invitation to visit the UCF and the entire Mathematics Department for its hospitality and very pleasant working environment. He is also particularly grateful to Professor Tommaso Russo for carefully reading the manuscript, for many valuable remarks related to well-separated unit sets and for the comments leading to a negative answer to the Question \ref{qtn:Q} in the last section of this paper.


\section{Preliminaries}\label{sec:2}
The notation used throughout this paper is quite standard and mostly follows \cite{FHHMZ}. For a Banach space $X$ we denote by $B_X$ the closed unit ball and $S_X$ its closed unit sphere. For an infinite set $N\subset \mathbb{N}$, $[N]^k$ stands for the $k$-element subsets of $N$. By $[N]^\omega$ and $[N]^{<\omega}$ we mean respectively the family of all infinite and finite subsets of $N$. Given sets $E, F\in[\mathbb{N}]^{<\omega}$, we shall use the notation $E\leq F$ (resp. $E<F$) to mean that $\max(E)\leq \min(F)$ (resp. $\max(E)< \min(F)$). In particular, $n\leq F$ means $\{n\}\leq F$. Recall that a basic sequence $(x_n)$ in $X$ is called seminormalized when there are constants $A, B>0$ so that $A\leq \| x_n\| \leq B$ for all $n\in \mathbb{N}$. A block basis of $(x_n)$ is a sequence of the form $z_n= \sum_{i\in E_n} \alpha^n_i x_i$, where $(\alpha^n_i)_{i\in E_n}\subset \mathbb{R}$ and $(E_n)$ is an increasing sequence of block of integers, that is, $E_n\in [\mathbb{N}]^{<\omega}$ and $E_n< E_{n+1}$ for all $n\in \mathbb{N}$. Every block basis of a basic sequence is a basic sequence. 

 
\medskip 
\section{Asymptotic monotone basic sequences} \label{sec:3}
Our first result is the following strengthening of Bessaga-Pe\l czy\'nski's selection principle.

\begin{lemma}\label{thm:M1} Let $X$ be an infinite dimensional Banach space. Then every seminormalized weakly null sequence in $X$ has an asymptotically monotone basic subsequence. 
\end{lemma}

\begin{proof} The proof involves combining Bessaga-Pe\l czy\'nski's selection approach together with a diagonal argument. Let $(y_i)$ be a seminormalized weakly null sequence in $X$ and fix $\varepsilon_n\searrow 0$ a decreasing null-sequence in $(0,1)$. We proceed by induction on $j=1, 2, 3, \dots$. 

Define $x_{n^0_i}= y_i$ for all $i\in \mathbb{N}$. Let $(x_{n^1_i})$ be  a $(1 + \varepsilon_1)$-basic subsequence of $(x_{n^0_i})$ given by Pe\l czy\'nski's selection principle (cf. \cite{Pel}). Accordingly, we have $n^1_1=1$ and 
\[
\Bigg\| \sum_{i=1}^m a_i x_{n^1_i}\Bigg\|\leq (1 + \varepsilon_1) \Bigg\| \sum_{i=1}^n a_i x_{n^1_i}\Bigg\|
\]
for all integers $m\leq n$ in $\mathbb{N}$ and for all scalars $(a_i)_{i=1}^\infty$ in $\coo$ (cf. \cite[p. 372]{Pel}). 

Now assume for some $j\geq 1$ that basic sequences $\big\{ (x_{n^k_i})_{i\in\mathbb{N}}\big\}_{k=1}^j$ have been already obtained so as to satisfy the next three properties, for all $k=1,\dots, j$: 
\begin{itemize}
\item $(x_{n^{k}_i})_{i\in \mathbb{N}}$ is a subsequence of $(x_{n^{k-1}_i})_{i\in \mathbb{N}}$;
\item $n^{k}_1=n^{1}_1$ if $k=1$,\, and $n^k_1= n^1_1,\, n^{k}_2= n^{2}_2,\,\dots, \,n^{k}_{k-1}= n^{k-1}_{k-1}$ for $k\geq 2$;
\item for all scalars $(a_i)_{i=1}^\infty$ in $\coo$, one has
\[
\Bigg\| \sum_{i=1}^{k} a_i x_{n^{k}_i}\Bigg\| \leq (1 + \varepsilon_{k}) \Bigg\| \sum_{i=1}^n a_i x_{n^{k}_i}\Bigg\|,\quad \forall \, n\geq k.
\]
\end{itemize}
Let $E=\spn\{ x_{n^{1}_1}, x_{n^{2}_2}, \dots, x_{n^{j-1}_{j-1}}, x_{n^{j}_{j}}\}$ and take a decreasing null sequence $(\delta_i)_i$ in $(0,1)$ satisfying
\[
\prod_{i=1}^\infty (1 + \delta_i)< 1+ \varepsilon_{j+1}.
\]
Clearly the sequence $(x_{n^j_i})_{i=1}^\infty$ is seminormalized and weakly null. From the proofs in \cite[Proposition and lemma]{Pel} we obtain an integer $\kappa(j,1)> j$ so that  
\[
\| e\| \leq (1 + \delta_{1})\big \| e + t x_{n^j_{\kappa(j,1)}}\big\|
\]
for all $e\in E$ and $t\in \mathbb{R}$. Let $\mathcal{K}=\prod_{i=2}^N(1 + \delta_i)$. By making an iterated use of the arguments developed in \cite[Proposition and lemma]{Pel} we obtain an increasing subsequence $(n^j_{\kappa(j,i)})_{i=1}^\infty$ of $\{ n^j_i\}_{i=1}^\infty$ such that
\[
\Bigg\| \sum_{i=1}^j a_i x_{n^i_i} + a_{j+1}  x_{n^j_{\kappa(j,1)}}\Bigg\| \leq \mathcal{K} \Bigg\|  \sum_{i=1}^j a_i x_{n^i_i} + a_{j+1}  x_{n^j_{\kappa(j,1)}}+ \sum_{i=j+2}^N a_i x_{n^j_{\kappa(j,i)}}\Bigg \|
\]
for all $N\geq j+2$ and for all scalars $(a_i)_{i=1}^\infty \in \coo$. We then define $(x_{n^{j+1}_i})$ as follows:
\[
x_{n^{j+1}_i}=\left\{
\begin{split}
&x_{n^i_i} \hskip 1.25cm \text{ for } i=1,\dots, j\\
&x_{n^j_{\kappa(j, i-j)}}  \hskip .4cm \text{ for } i\geq j+1.
\end{split}\right.
\]
The induction process succeeds. The above construction provides a countable family of basic subsequences  $(x_{n^k_i})_{k, i\in \mathbb{N}}$ of $(y_n)$ such that
\[
\Bigg\| \sum_{i=1}^k a_i x_{n^i_i}\Bigg\| \leq ( 1 + \varepsilon_{k}) \Bigg\| \sum_{i=1}^k a_i x_{n^i_i}  + \sum_{i=k+1}^\infty a_i x_{n^{k+1}_i}\Bigg\|
\]
for all $k\in \mathbb{N}$ and for all $(a_i)_{i=1}^\infty \in \coo$. 

Let us take now the diagonal sequence $(x_i)=(x_{n^i_i})$. Since $\{ n^i_i\}_{i=k+1}^\infty \subset \{ n^{k+1}_i\}_{i=1}^\infty$ for all $k\in \mathbb{N}$, we can make null those scalar terms $a_i$'s that are outside the diagonal and conclude directly from the previous inequality that $(x_i)$ is asymptotically monotone.
\end{proof}

\medskip 
\section{Symmetric $2$-separated sequences}\label{sec:4}
In connection with the ongoing study on the structure of well-separated subsets of the unit sphere of a Banach space the next question seems to be reasonable:

\begin{qtn}\label{qtn:1} Let $\varepsilon>0$. Under which conditions a Banach space $X$ admits a $(1+\varepsilon)$-equivalent renorming $\nn{\cdot}$ so that $(X, \nn{\cdot})$ has a normalized asymptotically monotone basic sequence being symmetrically $2$-separated?
\end{qtn}

A subset $A$ of a normed space is said to be symmetrically $\delta$-separated (see \cite{HKR}) when $\| x\pm y\| \geq \delta$ for any distinct elements $x, y\in A$ ($\delta>0$). Pioneer works along the study of well-separated sets in the unit sphere of Banach spaces include J. Elton and E. Odell \cite{EO} and C. Kottman \cite{K}. We refer the reader to \cite{HKR} for a recent and more detailed account about this topic. Motivation for Question \ref{qtn:1} comes from the problem of finding symmetrically $(1+\varepsilon)$-separated sequences of unit vectors under renorming techniques (see \cite[Section 5]{HKR}). In \cite[Theorem 7]{K} Kottman proved that every infinite-dimensional Banach space admits a renorming so that the new unit sphere contains a $2$-separated sequence.  In \cite[p.13]{HKR} the authors observed that Kottman's argument yields in fact a symmetrically $2$-separated sequence of norm-one vectors. However, the norm obtained in \cite{HKR} is only $2$-equivalent to the original one. Then, in \cite[Proposition 5.2]{HKR} they proved for every $\varepsilon>0$, that every Banach space $X$ admits an equivalent norm $\nn{\cdot}$ which is $(1+\varepsilon)$-equivalent to the original norm of $X$ and yet $S_{(X, \nn{\cdot})}$ contains an infinite symmetrically $(1+\delta)$-separated subset, for some $\delta>0$. They finally observed in \cite[Remark 5.3]{HKR} that every separable Banach space $X$ admits a strictly convex renorming $\nn{\cdot}$ so that the unit sphere of $S_{(X,\nn{\cdot})}$ contains no $2$-separated sequences. 

\smallskip 
Inspired by Proposition 5.1 in \cite{HKR} we remark the following.

\begin{prop}\label{prop:1sec4} Let $X$ be a Banach space and $Y$ a subspace of $X$. Assume that $\{ x_i; f_i\}_{i=1}^\infty$ is a biorthogonal system on $Y$ such that, for some  $\varepsilon\in (0,1)$, one has 
\[
\| x_i\| \leq 1 + \varepsilon\quad\text{and}\quad \sup_{i\neq j}\big( | f_i(y)| + | f_j(y)|\big)\leq (1 + \varepsilon) \| y\|
\]
for all $i\in \mathbb{N}$ and for all $y\in Y$. Then there is a $(1+\varepsilon)$-equivalent norm $\nn{\cdot}$ on $X$ such that $S_{(X, \nn{\cdot})}$ contains a symmetrically $2$-separated sequence. 
\end{prop}

\begin{proof} Indeed, to see this we first define a new norm on $Y$ by 
\[
|y| =\max\Bigg( \frac{1}{1 + \varepsilon}\| y\|, \sup_{i\neq j}\big( | f_i(y)| + | f_j(y)|\big)\Bigg)\quad\text{for }\, y\in Y. 
\]
Clearly we have
\[
\frac{1}{1 + \varepsilon}\|y\| \leq | y| \leq (1 + \varepsilon)\| y\|
\]
for all $y\in Y$, and hence $| \cdot|$ is $(1 + \varepsilon)$-equivalent to $\| \cdot\|$ on $Y$. Moreover, notice that $| x_i| = 1$ and $| x_i \pm x_j| \geq 2$ for all $i, j\in \mathbb{N}$ with $i\neq j$. Now, the proof of Lemma 2  in \cite{JO} yields an equivalent norm $\nn{\cdot}$ on $X$ satisfying $\nn{y}= | y|$ for all  $y\in Y$, and
\[
\frac{1}{(1 + \varepsilon)^2}\| x\|\leq \nn{ x}\leq (1 + \varepsilon)\| x\|\quad\text{for all } x\in X.
\]
It follows that $S_{(X, \nn{\cdot})}$ contains a symmetrically $2$-separated sequence. 
\end{proof}

\smallskip 
Proposition \ref{prop:1sec4} provides a guide for finding affirmative answers for Question \ref{qtn:1}. It is natural therefore to wonder which Banach spaces $X$ have the property: for every $\varepsilon>0$, there is a biorthogonal system $\{x_n; f_n\}_{n=1}^\infty$ in $(1+\varepsilon)B(Y)\times Y^*$, for some subspace $Y$ of $X$, so that the functionals $(f_n)_n$ satisfy 
\[
\sup_{i\neq j}\big( |f_i(x)| + |f_j(x)|\big)\leq (1+\varepsilon)\| x\|\quad\text{for all }\, x\in [(x_n)]. 
\]

\smallskip 
For the sake of simplicity,

\begin{defi} We shall call such a system as $(1+\varepsilon)$-biorthogonal system on $X$. 
\end{defi}

The next two results provide partial answers for the previous question when $X$ contains either isomorphic copies of $\ell_1$. 


\begin{prop}\label{prop:2sec4} Let $X$ be a Banach space containing a subspace isomorphic to $\ell_1$. Then for every $\varepsilon\in (0,1)$, $X$ admits a $(1+\varepsilon)$-biorthogonal system which generates a $(1+\varepsilon)$-equivalent norm $\nn{\cdot}$ under which the sphere $S_{(X, \nn{\cdot})}$ contains a bimonotone symmetrically $2$-separated basic sequence.
\end{prop}

\begin{proof} By James's non-distortion theorem \cite[Lemma 2.1]{J} there is an isomorphic embedding $T\colon \ell_1\to X$ such that $ \| x\|_{\ell_1}\leq \| Tx \| \leq (1 +\varepsilon)\|x \|_{\ell_1}$ for all $x\in \ell_1$, where $\| \cdot\|_{\ell_1}$ stands for the usual $\ell_1$-norm on $\ell_1$. Let $Y=T(\ell_1)$ and define a new norm on $Y$ as follows: 
\[
| y|_Y=\inf\{ \| x\|_{\ell_1} + \| y - Tx\| \colon x\in \ell_1\}\quad\text{ for } y\in Y. 
\]
It is not hard to verify that $T\colon \ell_1\to (Y, | \cdot|_Y)$ is an isometry. In addition to this, one has
\begin{equation}\label{eqn:0prop1}
\frac{1}{1+ \varepsilon}\| y\| \leq | y|_Y\leq \| y\| \quad\text{ for all } y\in Y. 
\end{equation}
Let now $(e_i)$ denote the canonical basis of $\ell_1$ and define $x_i:= T(e_i)$ for $i\in \mathbb{N}$. The fact that $T$ is an isometry from $\ell_1$ onto $(Y, | \cdot|_Y)$ clearly implies 
\begin{equation}\label{eqn:1prop1}
\Bigg| \sum_{i=1}^\infty a_i x_i\Bigg|_Y= \sum_{i=1}^\infty |a_i|\quad\text{ for all } (a_i)_{i=1}^\infty \in \coo. 
\end{equation}
Hence $(x_i)$ is a normalized bimonotone basis for $(Y, | \cdot|_Y)$. Furthermore, denoting $(f_i)$ the coefficient functionals of $(x_i)$ and using (\ref{eqn:0prop1}) in concert with (\ref{eqn:1prop1}) we conclude 
\[
 \sup_{i\neq j}\big( | f_i(y)| + | f_j(y)|\big)\leq  | y|_Y\leq \| y\|\quad\text{ for } y\in Y. 
\]
Notice that (\ref{eqn:0prop1}) also implies $\| x_i\|\leq (1 + \varepsilon)$ for all $i\in \mathbb{N}$. Thus $\{ x_n; f_n\}_{n=1}^\infty$ defines a $(1+ \varepsilon$)-biorthogonal system on $X$. By Proposition \ref{prop:1sec4} the result follows. 
\end{proof}



\section{Concluding remarks}\label{sec:6}
As we have pointed out in the introduction, Lemma \ref{thm:M1} is probably well-known to experts. We remark that part of our arguments concern a suitable use of \cite[Lemma]{Pel}, which can be seen as a {\it weakly null} version of the well-known Mazur's lemma. In turn, such a version (i.e. \cite[Lemma]{Pel}) can also be directly deduced from \cite[Lemma 1.5.1 and Proposition 1.5.4]{AK}. We also remark that Propositions \ref{prop:2sec4} strengthens \cite[Proposition 5.1]{HKR}  for the cases when $X$ contains $\ell_1$. The arguments presented in our approach were mostly inspired by the results in \cite{HKR}. 

It was remarked in \cite[Remark 5.3]{HKR} that every separable Banach space $X$ admits a strictly convex renorming $\nn{\cdot}$ so that $S_{(X, \nn{\cdot})}$ contains no $2$-separated sequences. In fact, one can even show that, for every $\varepsilon \in (0,1)$, every such $X$ admits a $(1+ \varepsilon)$-equivalent strictly convex norm for which the unit sphere contains no $2$-separated sequences. To see this take a dense sequence $\{ f_n\}_{n=1}^\infty \subset B_X$ such that $\| x\| =\sup\{ | f_n(x)| \colon x\in X\}$ for all $x\in X$. Next, define for $\delta\in (0,\sqrt{1 + \varepsilon} -1)$ a new norm on $X$ by $\nn{x}^2= \| x\|^2 + \delta \sum_{n=1}^\infty 2^{-n} f^2_n(x)$, $x\in X$. Direct calculation shows that $\nn{\cdot}$ is strictly convex and $(1+\varepsilon)$-equivalent to $\| \cdot\|$. Using then the same reasoning as in \cite[Remark 5.3]{HKR} one shows that $S_{(X, \nn{\cdot})}$ contains no $2$-separated sequences. In the opposite direction, one may naturally ask the question. 

\begin{qtn}\label{qtn:Q} Does every separable Banach space $X$ has a $(1 +\varepsilon)$-equivalent norm $\nn{\cdot}$ such $S_{(X, \nn{\cdot})}$ contains a symmetrically $2$-separated sequence?
\end{qtn}

Let us recall the {\it symmetric Kottman constant} introduced in \cite{CP}:
\[
K^s(X):=\sup\Big\{ \sigma> 0\,\colon \, \exists (x_n)_{n=1}^\infty \subset B_X \, \forall n\neq k\,\, \| x_n \pm x_k\| \geq \sigma\Big\}.
\]
The answer  for Question \ref{qtn:Q} is in general {\it no}. The main reason is that $K^s(X)$ is a continuous function on $X$, with respect to the Banach-Mazur distance (see the proof of Theorem 7 in \cite{K}). Hence, if we pick $X$ so that $K^s(X)<2$ (e.g., $\ell_p$ for $1< p< \infty$) every $(1+\varepsilon)$-renorming, with $\varepsilon$ small enough, will have symmetric Kottman constant less than $2$. Therefore, Question \ref{qtn:Q} has positive answer for $X$ exactly when $K^s(X)=2$. 

The problem of describing the class of Banach spaces $X$ for which $K^s(X)=2$ has been considered by several authors. In \cite{HKR} the authors showed that if $X$ contains a spreading model isomorphic to $\ell_1$ then $K^s(X)=2$. The corresponding problem for $\co$ spreading models was left as an open problem \cite[Problem 4.11]{HKR}. It is not hard to show that $K^s(T^*)=2$, where $T^*$ is the {\it original} Tsirelson space (cf. \cite{FJ} for basic definition and properties of Tsirelson's spaces). On the other hand, it is a well-known fact that every spreading model of $T^*$ is isomorphic to $\co$. These facts led us to believe that $K^s(X)=2$ for every Banach space $X$ that admit $\co$ as a spreading model. We observe, however, that the assumption on a Banach space $X$ to admit a spreading model isomorphic to $\co$ is not sufficient to imply that $K^s(X)=2$. We thank Professor Tommaso Russo for informing us privately about this and many other of his recent achievements on the matter.


\nocite{*}

\begin{thebibliography}{99}

\bibitem[AK]{AK} F. Albiac and N. Kalton, Topics in Banach space theory, Graduate Texts in Mathematics, {\bf 233}. Springer, New York, 2006.



\bibitem[BP]{BP} C. Bessaga and A. Pe\l czy\'nski, On bases and unconditional convergence of series in Banach spaces. Studia Math. {\bf 17} (1958) 151--164.

\bibitem[CP]{CP} J. M. F. Castillo and P. L. Papini, On Kottman's constant in Banach spaces, Function Spaces IX, Banach Center Publ. {\bf 92} (2011) 75--84. 

\bibitem[Day]{Day} M. M. Day, On the basis problem in normed spaces, Proc. Amer. Math. Soc. {\bf 13} (1962) 655--658. 

\bibitem[D]{D} J. Diestel, Sequences and Series in Banach spaces, Graduate Texts in Mathematics, 92, Springer-Verlag, New York, 1984. 

\bibitem[DLT]{DLT} P. N. Dowling, C. J. Lennard and B. Turett, Some fixed point results in $\ell_1$ and $\co$, Nonlinear Anal. {\bf 39} (2000) 929--936. 

\bibitem[EO]{EO} J. Elton and E. Odell, The unit ball of every infinite-dimensional normed linear space contains a $(1+\varepsilon)$-separated sequence, Colloq. Math. {\bf 44} (1981) 105--109. 

\bibitem[FHHMZ]{FHHMZ} M. Fabian, P. Habala, P. H\'ajek, V. Montesinos, V. Zizler, Banach Space Theory: The Basis for Linear and Nonlinear Analysis. CMS Books in Mathematics, Springer 2011. 

\bibitem[FJ]{FJ} T. Figiel and W. B. Johnson, A uniformly convex Banach space which contains no $\ell_p$, Composition Math. {\bf 29} (1974) 179--190. 

\bibitem[HKR]{HKR} P. H\'ajek, T. Kania, T. Russo, Symmetrically separated sequences in the unit sphere of a Banach space, J. Functional Analysis. 2018. 

\bibitem[HSVZ]{HSVZ} P. H\'ajek, V. M. Sataluc\'ia, J. Vanderwerff, V. Zizler, Biorthogonal Systems in Banach Spaces. CMS Books in Mathematics, Springer 2006. 

\bibitem[J]{J} R. C. James,  Uniformly non-square Banach spaces, Ann. of Math. {\bf 80} (1964), 542--550.  

\bibitem[JO]{JO} W. B. Johnson and E. Odell, The diameter of the isomorphism class of a Banach space, Ann. of Math. {\bf 162} (2005) 423--437. 

\bibitem[K]{K} C. A. Kottman, Subsets of the unit ball that are separated by more than one, Studia Math. {\bf 53} (1975), 15--27. 

\bibitem[M]{M} V. D. Milman, Geometric theory of Banach spaces I: Theory of basic and minimal systems. Uspehi Mat. Nauk {\bf 25} (1970) 113--173. 

\bibitem[MV]{MV} J. Mujica and D. M. Vieira, Schauder basis and the bounded approximation property in separable Banach spaces, Studia Math. {\bf 196} (2010) 1--12. 

\bibitem[OS]{OS} E. W. Odell and T. Schlumprecht, Distortion and asymptotic structure, Handbook of Banach spaces II, Editors W. B. Johnson and J. Lindenstrauss, Elsevier, 2003, 1333--1360. 

\bibitem[Pel]{Pel} A. Pe\l czy\'nski, A note on the paper of I. Singer "Basic sequences and reflexivity of Banach spaces", Studia Math. {\bf 21} (1962) 371--374.

\bibitem[Si]{Si} I. Singer, Geometric properties of the norm and basic sequences in Banach spaces, Bull. Austral. Math. Soc. {\bf 13} (1975) 325--335. 

\bibitem[Si2]{Si2} I. Singer, On the problem of non-smoothness of non-reflexive second conjugate spaces, Bull. Austral. Math. Soc. {\bf 12} (1975) 407--416. 

\bibitem[Si3]{Si3} I. Singer, Bases in Banach spaces II, Springer-Verlag Berlin Heidelberg New York, 1981. 



\end{thebibliography}

\end{document}